\documentclass[3p]{elsarticle}
\usepackage{amscd, amssymb, amsthm}
\usepackage[all]{xy}
\usepackage{color}
\usepackage{url}
\usepackage{amscd, amssymb, amsfonts, amsthm, amsmath}
\usepackage{mathdots}
\usepackage[breaklinks=true]{hyperref}
\usepackage{multicol}
\usepackage{mathrsfs}
\usepackage{yfonts}

\usepackage[normalem]{ulem}

\usepackage[labelsep=endash]{caption}

\numberwithin{equation}{section}
\usepackage[labelsep=endash]{caption}

\makeatletter
\def\ps@pprintTitle{%
  \let\@oddhead\@empty
  \let\@evenhead\@empty
  \let\@oddfoot\@empty
  \let\@evenfoot\@oddfoot
}
\makeatother

\newcommand{\soc}{\mbox{\rm{Soc}}}
\newcommand{\topp}{\mbox{\rm{Top}}}

\def\ker{\mbox{\rm{Ker}}}

\def\mod{\mbox{\rm{mod}}}
\def\Mod{\mbox{\rm{Mod}}}
\def\ind{\mbox{\rm{ind}}}
\def\add{\mbox{\rm{add}}}

\def\l{\mbox{\rm{length}}}

\def\rad{\hbox{\rm{rad}}}

\def\pd{\mbox{\rm{pd}}}

\newcommand{\fin}{\mbox{\rm{findim}}}
\newcommand{\Fin}{\mbox{\rm{Findim}}}

\newcommand{\Dell}{\mbox{\rm{dell}}}
\newcommand{\Dellm}{\mbox{\rm{Dell}}}
\newcommand{\DDell}{\mbox{\rm{ddell}}}

\def\st{\mbox{\rm{s}}}
\def\tg{\mbox{\rm{t}}}
\def\supp{\mbox{\rm{supp}}}

\begin{document}

\newcommand{\mono}[1]{%
\gdef\puA{#1}}
\newcommand{\puA}{}
\newcommand{\faculty}[1]{%
\gdef\puC{#1}}
\newcommand{\puC}{}
\newcommand{\facultad}[1]{%
\gdef\puD{#1}}
\newcommand{\puD}{}
\newcommand{\N}{\mathbb{N}}
\newcommand{\Z}{\mathbb{Z}}
\newtheorem{teorema}{Theorem}[section]
\newtheorem{teo}{Theorem}[section]
\newtheorem{prop}[teo]{Proposition}
\newtheorem{lema}[teo] {Lemma}
\newdefinition{ej}[teo]{Example}
\newtheorem{obs}[teo]{Remark}
\newtheorem{defi}[teo]{Definition}
\newtheorem{coro}[teo]{Corollary}
\newtheorem{nota}[teo]{Notation}



\title{Delooping levels on Morita context algebras}

\author{Marcos Barrios}
\ead{marcosb@fing.edu.uy}

\author{Marcelo Lanzilotta\corref{cor1}}
\ead{marclan@fing.edu.uy}

\author{Gustavo Mata\corref{cor2}}
\address{Universidad de la Rep\'ublica, Facultad de Ingenier\'ia -  Av. Julio Herrera y Reissig 565, Montevideo, Uruguay} 
\ead{gmata@fing.edu.uy}

\cortext[cor1]{The author was supported by the project of  MathAmSud - 23-MATH-06.}
\cortext[cor2]{Corresponding author}

\begin{abstract}{ } 
We study the behavior of various delooping levels in the setting of Morita context algebras with zero bimodule homomorphisms. Building on the notion of syzygies for modules over such algebras, we establish conditions under which the delooping level (see \cite{G1}) of a Morita context algebra is finite, provided that the component algebras also have finite delooping levels. We further extend these results to other invariants, the global delooping level (see \cite{BLM}) and the derived delooping level (see \cite{GI}), showing how their finiteness is preserved in the Morita context. Several examples illustrate both the applicability of our criteria and the necessity of the hypotheses, highlighting cases where the delooping level of the whole algebra can be smaller than that of its components.
\end{abstract} 


\maketitle

\section{Introduction}

The finitistic dimension conjectures were proposed by H. Bass in 1960 as problems in the context of noetherian rings. The conjectures assert that:

$$\fin(R) = \Fin (R)\ \mbox{and}\ \   \fin(R) < \infty, $$

\noindent where $\Fin(R) = \sup\{\pd(M): M \in \Mod R,\ \pd(M)<\infty\}$ and $\fin(R) = \sup\{\pd(M): M \in \mod R,\ \pd(M)<\infty\}$.\\

For Artin algebras, the first conjecture is known to be false, B. Huisgen-Zimmermann constructed monomial algebras for which it fails. In relation with the second one there is a possible counterexample provided in October 2026 by OpenAI (\url{https://github.com/openai/math/tree/main/preprints}).\\

The delooping level for noetherian rings, denoted by $\Dell$, was introduced by Vincent Gélinas in \cite{G1}. He proved for a noetherian semiperfect ring $R$ that $\Fin (R^{op}) \leq \Dell(R)$.	The authors showed in \cite{KR} that the delooping level of an Artin algebra can be infinite, with $\Fin$ finite.\\
	
The derived delooping level, denoted by $\DDell$, was introduced by Ruoyu Guo and Kiyoshi Igusa in \citep{GI}. They proved for an Artin algebra $\Gamma$ that $$\Fin (\Gamma^{op}) \leq \DDell(\Gamma) \leq \Dell(\Gamma).$$ Thus, compared to the delooping level, the derived delooping level yields a sharper bound on the finitistic dimension. \\

Let $A$ and $B$ be Artin algebras, $M$ an $A$-$B$-bimodule, and $N$ a $B$-$A$-bimodule. Let
$$\alpha : N \otimes_A M \rightarrow B \text{, and } \beta : M \otimes_B N \rightarrow A$$ be bimodule homomorphisms. The associated Morita context $\mathcal{M}=(A, M, N, B, \alpha, \beta)$, gives rise to the Morita context algebra (see \cite{A}):
$$ \Lambda_{(\alpha, \beta)} = \begin{pmatrix}
A & M\\
N & B
\end{pmatrix},$$
where the multiplication in $\Lambda_{(\alpha, \beta)}$ is given by 
$$\begin{pmatrix}
a & m\\
n & b
\end{pmatrix} \cdot \begin{pmatrix}
a' & m'\\
n' & b'
\end{pmatrix} = \begin{pmatrix}
aa'+\beta(m\otimes n') & am'+mb'\\
na' + bn' & bb' + \alpha(n \otimes m')
\end{pmatrix}$$
where the maps $\alpha$ and $\beta$ satisfy $\alpha(n \otimes m)n' = n\beta(m \otimes n' )$ and $m\alpha(n \otimes m') = \beta(m \otimes n)m' $, to make $\Lambda_{(\alpha, \beta)}$ associative.\\ 

A Morita context algebra is said to have zero bimodule homomorphisms if $\alpha = \beta = 0$. One advatage for working in this setting is that the module categories Mod \!$A$, Mod\! $B$, Mod $\begin{pmatrix}
	A & 0\\
	N & B
\end{pmatrix},$ and  Mod $\begin{pmatrix}
A & M\\
0 & B
\end{pmatrix},$ are fully embedded into the module category Mod \!$\Lambda_{(0, 0)}$.

This family of  algebras has been extensively studied from the homological point of view. We might mention, among other works: \cite{BM2}, \cite{BM3}, \cite{CRS}, \cite{GP*}, \cite{GP}, \cite{GX} and \cite{LV}.

In this paper, we study homoligcal invariants, as the delooping level, derived delooping level, etc., for Morita context algebras which have zero bimodule morphisms.
We now give a brief outline of the paper. All necessary preliminaries are presented in Section 2. In Section 3, we describe and compute the syzygies for module categories over Morita context algebras. The main result, Theorem \ref{main theorem}, appears in Section 4.\\

\noindent {\bf Main result} (Theorem \ref{main theorem})
Consider $A$ and $B$ Artin algebras,  with finite delooping level.  If
$ \Lambda_{(0, 0)} = \begin{pmatrix}
	A & M\\
	N & B
\end{pmatrix},$
with $_AM$ and $_BN$ semisimples modules, and there exists $n\in \mathbb{N}$ such that $\Omega_{\Lambda}^n((\oplus_{a \in \mathcal{E}_A}  \tilde{S}_a )\oplus (\oplus_{b \in \mathcal{E}_B}  \tilde{S}_b))$ is infinitely deloopable (see Lemma \ref{sizigia}), then $\Dell(\Lambda) < \infty$. \\

Also in this section, we provide various families of algebras in which the hypothesis of the above theorem is satisfied. Section 5 is devoted to study other homological invariants related to delooping level. Finally, Section 6 is devoted to present several examples.
	
\section{Preliminaries}

Throughout this article, $A, B, \Lambda, \Gamma$ are Artin algebras. For an Artin algebra $\Gamma$,  $\Mod \Gamma$ is the category of right $\Gamma$-modules, and $\mod \Gamma$ is the full subcategory of $\Mod \Gamma$ determined by the finitely generated right $\Gamma$-modules. Also $\ind \Gamma$ is the full subcategory of $\mod \Gamma$ formed by all the indecomposable modules, and $\mathcal{P}_{\Gamma}$ is the class of projective $\Gamma$-modules. 

We use $\mathcal{S}_{\Gamma}$ for the set of iso-classes of simple $\Gamma$-modules, $\Gamma_0 = \oplus_{S \in \mathcal{S}_{\Gamma}}S = \frac{\Gamma}{\rad \Gamma}$, where $\rad \Gamma$ is the Jacobson radical of $\Gamma$. If $M \in \mod \Gamma$, $\soc(M)$ denotes its socle. We denote by $\mathcal{E}_{\Gamma} = \{e_1,\ldots, e_n\}$ a complete set of primitive orthogonal idempotents of $\Gamma$. 

For $M\in \mod \Gamma$ we denote by $P(M)$ its projective cover, by $\Omega(M)$ its syzygy, by $\rad M$ its radical and by $M^k = \displaystyle \oplus_{i=1}^k M$. For $M, T\in \mod \Gamma$  we denote $M  \stackrel{\oplus}{\hookrightarrow} T$ if there exists $N\in \mod \Gamma$ such that $M\oplus N = T$.

\begin{teo}(Theorem 1 from \cite{S}) \label{radical}
If $\mathcal{M}=(A, M, N, B, \alpha, \beta)$ is a Morita context and $ \Lambda_{(\alpha, \beta)} = \begin{pmatrix}
A & M\\
N & B
\end{pmatrix},$ the radical of the algebra $ \rad \Lambda_{(\alpha, \beta)} = \begin{pmatrix}
\rad A & M'\\
N' & \rad B
\end{pmatrix},$ where  $M'=\{ m \in  M : Nm \subset \rad B\}$, and $N' = \{n \in N : Mn \subset \rad A \}$.
\end{teo}

In case $\alpha=0=\beta$ then: $$ \rad \Lambda_{(0, 0)} = \begin{pmatrix}
	\rad A & M\\
	N & \rad B
\end{pmatrix}.$$

\vspace{0.4cm}


In what follows, we shall review the definitions of some homological measures defined recently.  

\begin{defi}(\cite{G1})
Let $\Gamma$ be an Artin algebra.
\begin{itemize}
\item[-] The \textbf{delooping level} of a $\Gamma$-module $X$ is defined as
$$\Dell(X) = \inf \{n \in \mathbb{N} : \Omega^n(X) \text{ is a direct summand of } \Omega^{n+1}(Y) \text{ in } \underline{\mod} \Gamma \text{, for some } Y\in \mod \Gamma\}.$$
\item[-] If no such $n$ exist, $\Dell(X) = \infty$.
\end{itemize}
\end{defi}
 
 \vspace{0.3cm}

\noindent Observe that for a module $X \in \mod \Gamma$, then $\Dell(X) = 0 $ if and only if $X \hookrightarrow P$, with $P \in \mathcal{P}_{\Gamma}$. In particular if $S$ is a semisimple module $\Dell (S) = 0$ if and only $S \stackrel{\oplus}{\hookrightarrow} \soc(P)$, with $P \in \mathcal{P}_{\Gamma}$.\\

\begin{defi} (\cite{GI})
Let $\Gamma$ be an Artin algebra.
\begin{itemize}

\item[-] Given $k \in \mathbb{Z}^+$, the \textbf{$k$-delooping level} of a $\Gamma$-module $X$ is defined as 
$$k-\Dell(X) = \inf \{n \in \mathbb{N} : \Omega^n(X) \stackrel{\oplus}{\hookrightarrow}  \Omega^{n+k}(Y)\text{ in } \underline{\mod} \Gamma \text{, for some } Y \in \mod \Gamma\}.$$
If no such $n$ exist, $k-\Dell(X) = \infty$.

\item[-] The $\Gamma$-module $X$ is \textbf{infinitely deloopable} if $k-\Dell(X) = 0$, $\forall k \in \mathbb{Z}^+$. 

\end{itemize}
\end{defi}

\vspace{0.3cm}

\noindent Let’s summarise a few observations about infinitely deloopable modules:

\begin{obs}\label{318r} 
	\begin{enumerate}
		\item 
First observe that $X \in \mod \Gamma$ is infinitely deloopable if and only if for each $n\in \mathbb{Z}^+$ there is a module $X_n \in \mod \Gamma$ and a projective module $P_n$ such that $X  \stackrel{\oplus}{\hookrightarrow} \Omega^n(X_n) \oplus P_n $. That is, infinitely deloopable indecomposable modules are those indecomposable modules in add$(\displaystyle \bigcap_{n\in \N} \Omega^n (\mod \Gamma))$.

\item 
 Notice also that infinitely deloopable (not projective) modules have infinite injective dimension.
\item 
Next observe that if $X$ is infinitely deloopable, then $\Omega^{n}(X)$ is infinitely deloopable, for each $n\in \N$.
\item 
Note that for any finite set $\mathcal{T}=\{T_1, T_2, \cdots,T_m\}$ of indecomposable $\Gamma$-modules, there exists $k\in \N$ such that $\mathcal{T}=\mathcal{T}_0 \sqcup \mathcal{T}_{\infty}$ with:
\begin{center}
$\mathcal{T}_0\ \cap$ add$\displaystyle (\bigcup_{h\geq k} \Omega^h (\mathcal\mod \Gamma))=\emptyset$ 
and  $\mathcal{T}_{\infty}\subset $ $\displaystyle  \bigcap_{j\geq k}$add$\displaystyle (\bigcup_{h\geq j} \Omega^h (\mathcal \mod \Gamma)) = $ add$(\displaystyle \bigcap_{n\in \N} \Omega^n (\mod \Gamma))$.
\end{center}

	\end{enumerate}
\end{obs}

\begin{defi} (\cite{GI})
	Let $\Gamma$ be an Artin algebra.\\
The \textbf{derived delooping level} of a $\Gamma$-module $X$ is defined as:
$$\DDell (X) = \inf\{m \in \mathbb{N} : \text{ there exists } \ n \leq m \text{ and an exact sequence in } \mod \Gamma $$ $$ 0 \rightarrow C_n \rightarrow \ldots \rightarrow C_0 \rightarrow X \rightarrow 0,\  \text{ where } (i+1)-\Dell (C_i) \leq  m-i,\ i = 0, 1, \ldots , n\}.$$
\end{defi}

\vspace{0.3cm}

\noindent Observe that given $X \in \mod \Gamma$,
 $\DDell(X) = 0 $ if and only if $\Dell(X)=0$, see Proposition \ref*{DDell=0}.

\vspace{0.3cm}

Let us recall some older homological definitions, given by K. Igusa and D. Zacharia in \cite{IZ} (see also \cite{IZ0}). 

\begin{defi}(\cite{IZ})
	Let $\Gamma$ be an Artin algebra. A module $X\in \mod \Gamma$ is
	
	\begin{itemize}
		
		\item[-] periodic if $X  \stackrel{\oplus}{\hookrightarrow} \Omega^k(X)$ for some $k\in \mathbb{N}^*$. The smallest $k$ is called the period of $X$;
		
		\item[-] virtually periodic if $X  \stackrel{\oplus}{\hookrightarrow} \Omega^k(X')$ for some periodic $X'$ and $k\in \mathbb{N}$.
		
	\end{itemize} 
	
\end{defi}
\vspace{0.2cm}

From the previous definitions, following remark is straightforward:

\begin{obs}\label{31/8}
Let $\Gamma$ be an Artin algebra and $X \in \mod \Gamma$.
\begin{itemize}

\item[-]  If $X$ is periodic or virtually periodic, then $k-\Dell(X) = 0$, $\forall k \in \mathbb{N}$, and therefore $X$ is infinitely deloopable. 

\item[-]  If $\Omega^{n}(X)$ is virtually periodic, then $k-\Dell(X) \leq n$, $\forall k \in \mathbb{N}$.

\end{itemize}
\end{obs}

\vspace{0.3cm}

\begin{defi}(\cite{BLM}, \cite{G1}, \cite{GI})
For an Artin algebra $\Gamma$, we define:
\begin{itemize}
\item[-] its delooping level (\cite{G1})  as 
$\Dell(\Gamma) = \sup \{\Dell(S): S \in \mathcal{S}_{\Gamma}\}$;

\item[-] its $k$-delooping level (\cite{GI}) as 
$k-\Dell(\Gamma) = \sup \{k-\Dell(S): S \in \mathcal{S}_{\Gamma}\}$;

\item[-] its derived delooping level (\cite{GI}) as 
$\DDell(\Gamma) = \sup \{\DDell(S): S \in \mathcal{S}_{\Gamma}\}$;

\item[-] its global delooping level (\cite{BLM}) as
$\Dellm(\Gamma) = \sup \{\Dell(M): M \in \mod \Gamma\}$.

\end{itemize}
\end{defi}

\vspace{0.3cm}
The following example (\cite{KR}) shows that for an Artin algebra $\Gamma$, its delooping level can be infinite.

\begin{ej}(\cite{KR})\label{ejemplo infinito}
Let $\Bbbk$ be a field with $q \in \Bbbk$ an element with infinite multiplicative order.
Consider $\Gamma(q)= \frac{\Bbbk Q}{I_q}$, the algebra given by:
$$Q :  \xymatrix{1 \ar[r]^{\delta}  & 2 \ar@(ru,rd)^{\beta} \ar@(lu,ru)^{\alpha} \ar@(ld,rd)_{\gamma}}$$
and the ideal $I_q = \langle \alpha^2, \beta^2, \gamma^2, \gamma \beta, \beta \alpha + q\alpha \beta, \gamma \alpha - \alpha \gamma, \beta \gamma - \alpha \gamma, \delta \alpha - q\delta \beta, F^3 \rangle $, where $F$ is the ideal generated by the arrows.\\

In \cite{KR} the authors showed that $\Dell(S_2) = 0$ and $\Dell(S_1) = \infty$, hence $\Dell(\Gamma) = \infty$.

\end{ej}

\begin{teo} \label{Igualdad}
If $\Gamma$ is an Artin algebra from one of the following families of algebras:

\begin{itemize}
\item[-] Gorenstein algebras (\cite{G1}),

\item[-] radical square zero algebras (\cite{G2}),

\item[-] truncated path algebras (\cite{BLM}), 
\end{itemize}

then $\Fin({\Gamma}^{op}) = \Dell(\Gamma)$.
\end{teo}

However there exist examples of Artin algebras (see \cite{BLM}) where $\Fin({\Gamma}^{op}) \not = \Dell(\Gamma)$.

\begin{teo} (\cite{GI})\label{GUO&IGUSA} For an Artin algebra $\Gamma$, we have
 $$\Fin (\Gamma^{op}) \leq  \DDell (\Gamma) \leq \Dell (\Gamma).$$
\end{teo}

As a consequence of Theorems \ref{Igualdad} and \ref{GUO&IGUSA}, is obtained the following result.

\begin{coro}
Let $\Gamma$ be an Artin algebra from one of the following families:

\begin{itemize}
\item[-] Gorenstein algebras,

\item[-] radical square zero algebras,

\item[-] truncated path algebras.
\end{itemize}

Then $\DDell(\Gamma)= \Fin({\Gamma}^{op})$.
\end{coro}

\section{Syzygies on Morita context algebras}	

In this section we investigate the structure of syzygies of modules over Morita context algebras, based on previous works as \cite{BM2}, \cite{GP}, and  \cite{LV}. We introduce key lemmas and remarks about how syzygy modules decompose in these algebras. These technical results provide the foundation for analyzing delooping levels in the following sections.\\

\noindent Let $\Lambda$ be an Artin algebra with $\mathcal{E}_A=\{e_1, \ldots, e_n\}$ a complete set of primitive orthogonal idempotents. For $X \in \mod \Lambda$, we recall that $\supp(X)$ is the subset of $\mathcal{E}_A=\{e_1, \ldots, e_n\}$ determined by the idempotents $e_j$ such that $Xe_j \not = \{0\}$.

\begin{obs}\label{obssoporte}
Let $X,Y\in  \mod \Lambda$, and let $f\in $Hom$_{\Lambda}(X,Y)$, then $\supp(f(X)) \subset \supp(X) \cap \supp(Y)$.
\end{obs}

\begin{lema}\label{lema1}
Let $\Lambda$ be an Artin algebra, and let $X,Y\in  \mod \Lambda$ such that $\supp(X) \cap \supp(Y) = \emptyset$. If $M$ is an indecomposable submodule of $X\oplus Y$, then either $M$ is a submodule of $X$ or a submodule of $Y$.
\begin{proof}

Suppose that $X = \displaystyle \bigoplus_{i=1}^{i=k} Xe_i $ and $Y = \displaystyle \bigoplus_{i=k+1}^{i=n} Ye_i$, as vectorial spaces. Consider  $M = M_X \oplus M_Y$ as vectorial spaces, where $M_X = Me_1 \oplus Me_2 \oplus \ldots \oplus Me_k \subset X$ and $M_Y = Me_{k+1}\oplus Me_{k+2} \oplus \ldots \oplus Me_n \subset Y$. Let $\alpha \in \Lambda$, $m \in M_X$ and consider $m\cdot \alpha \in M$. Assume $m\cdot \alpha   = x + y$ where $x \in M_X \subset X$ and $y\in M_Y \subset Y$. But $m \in M_X\subset X$, then $m\cdot \alpha \in X$ and then $y = 0$. So we deduce that $M_X$ is a submodule of $M$. Analogously, we also deduce that $M_Y$ is a submodule of $M$. Hence $M = M_X\oplus M_Y$ is a direct sum of $\Lambda$-modules, but $M$ is an indecomposable module. Therefore $M_X=(0)$ or $M_Y=(0)$. Then $M\subset X$ or $M\subset Y$.
%
%
\end{proof}
\end{lema}

\begin{lema}\label{lema2}
Let $\Lambda_{(\alpha,\beta)} = \begin{pmatrix}
A & M \\
N & B\\
\end{pmatrix}$ be a Morita context algebra. If $0 \rightarrow X \stackrel{f}{\rightarrow} Y_1 \oplus Y_2\stackrel{\pi}{\rightarrow} Z \rightarrow 0$ is a short exact sequence such that:

\begin{enumerate}
\item $X$, $Y_1 \in \mod A$, $Y_2, Z \in \mod B$, then $X \cong Y_1$ and $Y_2 \cong Z$;

\item $X$, $Y_1 \in \mod B$, $Y_2, Z \in \mod A$, then $X \cong Y_1$ and $Y_2 \cong Z$.
\end{enumerate}

\begin{proof}
We only prove the first item, the second one is dual. Since $\supp(Y_1) \cap \supp(Y_2) = \{0\}$, by Lemma \ref{lema1} $X \underset{f}{\subset} Y_1$. On the other hand by Remark \ref{obssoporte} $\supp(\pi(Y_1)) = \emptyset$, then $Y_1 \subset \ker (\pi)= $Im$(f)\cong X $, then $\l(X) = \l(Y_1)$ and $X \cong Y_1$. Hence, $f$ is a split monomorphism and we deduce that $Y_2 \cong Z$.
\end{proof}
\end{lema}

\vspace{0.2cm}

\begin{obs}\label{remark idempotentes}
Given $A$ and $B$ Artin algebras, $ _AM_B$ and $ _BN_A$ bimodules, in case we set both bimodule homomorphisms $\alpha$ and $\beta$ as zero, considering the Artin algebra $\Lambda_{(0,0)} = \begin{pmatrix}
A & M \\
N & B\\
\end{pmatrix}$, as we have already seen,  $\mod A \hookrightarrow \mod \Lambda_{(0,0)} $ and  $\mod B \hookrightarrow \mod \Lambda_{(0,0)} $ are full embeddings of categories. Consider $\mathcal{E}_A=\{e_1, \ldots, e_n\}$ and $\mathcal{E}_B=\{f_1, \ldots, f_m\}$ two complete sets of primitive orthogonal idempotents of $A$ and $B$ respectively, then $$\mathcal{E}_{\Lambda} = \left\{ \begin{pmatrix}
e_1 & 0 \\
0   & 0 
\end{pmatrix}, \ldots, \begin{pmatrix}
e_n & 0 \\
0   & 0 
\end{pmatrix}, \begin{pmatrix}
0 & 0 \\
0 & f_1 
\end{pmatrix}, \ldots, \begin{pmatrix}
0 & 0 \\
0   & f_m 
\end{pmatrix}  \right\}$$ is a complete set of primitive orthogonal idempotents of $\Lambda$ (for short $\mathcal{E}_{\Lambda} = \{e_1, \ldots, e_n, f_1, \ldots , f_m\}$). 

Since $0 \rightarrow \begin{pmatrix}
0 & M\\ N & B 
\end{pmatrix} \rightarrow \Lambda \stackrel{\pi_A \ \ }{\rightarrow A} \rightarrow 0$ and $0 \rightarrow \begin{pmatrix}
A & M\\ N & 0 
\end{pmatrix} \rightarrow \Lambda \stackrel{\pi_B \ \ }{\rightarrow B} \rightarrow 0$ are exact sequences of right $\Lambda$-modules, then for $e_k, f_j \in \mathcal{E}_{\Lambda}$, we have the short exact sequences in $\mod \Lambda$

\begin{equation}\label{eq 1}
0\rightarrow e_kM \rightarrow e_k \Lambda \rightarrow e_k A \rightarrow 0 \text{, and } 0\rightarrow f_j N \rightarrow f_j \Lambda \rightarrow f_j B \rightarrow 0
\end{equation}
where $e_kM \in \mod B$ and $f_jN \in \mod A$.
\end{obs}

\vspace{0.2cm}

\begin{defi}\label{petite}
	Given $A$ and $B$ Artin algebras, $ _AM_B$ and $ _BN_A$ bimodules, we say that the Morita context algebra $\Lambda_{(0,0)} = \begin{pmatrix}
		A & M \\
		N & B\\
	\end{pmatrix}$ is of petite connection type if $_AM$ and $_BN$ are semisimple modules.
\end{defi}

Now, by Theorem \ref{radical}, we have the following remark.

\begin{obs}
If $\Lambda_{(0,0)} = \begin{pmatrix}
A & M \\
N & B\\
\end{pmatrix}$ is a Morita context algebra of petite connection type, then $$\rad \Lambda = \begin{pmatrix}
\rad A & M \\
N & \rad B\\
\end{pmatrix} = \begin{pmatrix}
\rad A & 0 \\
0 & 0\\
\end{pmatrix} \oplus \begin{pmatrix}
0 & 0 \\
N & 0\\
\end{pmatrix} \oplus \begin{pmatrix}
0 & M \\
0 & 0\\
\end{pmatrix}  \oplus \begin{pmatrix}
0 & 0\\
0 & \rad B\\
\end{pmatrix}$$ as $\Lambda$-modules,  where the first two direct summands are $A$-modules and the last two are $B$-modules.
\end{obs}

\vspace{0.2cm}

\begin{lema} \label{sizigia}
Given $\Lambda_{(0,0)} = \begin{pmatrix}
	A & M \\
	N & B\\
\end{pmatrix}$ a Morita context algebra of petite connection type, for any $X_A \in \mod A $ and $X_B \in \mod B$, we get that  $\Omega_{\Lambda}( X_A) = \Omega_A( X_A) \oplus \tilde{X}_B$ and $\Omega_{\Lambda}( X_B) = \Omega_B( X_B) \oplus \tilde{X}_A$, where $\tilde{X}_A \in \mod A$ and $ \tilde{X}_B \in \mod B$.  
\begin{proof}

Consider the following short exact sequences:

\begin{itemize}
\item[] $$0 \rightarrow \Omega_A (X_A) \rightarrow P_A(X_A) \rightarrow X_A\rightarrow 0, \mbox{ in } \mod A,$$

\item[] $$0 \rightarrow \Omega_{\Lambda} (X_A) \rightarrow P_{\Lambda}(X_A) \rightarrow X_A\rightarrow 0, \mbox{ in } \mod \Lambda.$$

\end{itemize}

From the previous exact sequences, we construct, by using Remark \ref{remark idempotentes}, the commutative diagram (in $\mod \Lambda$) with exact columns and rows
$$\xymatrix{  &&0\ar[d] & 0 \ar[d] & \\  & 0 \ar[d] \ar[r]  & \tilde{X}_B \ar[d]\ar[r]& \Omega_{\Lambda}(X_A)\ar[d] & \\   &  0 \ar[r] \ar[d] & P_{\Lambda}(X_A) \ar[d] \ar[r]^{1}& P_{\Lambda}(X_A)\ar[d] \ar[r]& 0  \\0 \ar[r] &\Omega_{A}
(X_A) \ar[r] \ar[d]^{1} & P_A(X_A) \ar[d] \ar[r] & X_A \ar[d] \ar[r] & 0 \\ &\Omega_{A}
(X_A) \ar[d] & 0 & 0 & \\ &0&&&}$$

By the Snake's Lemma, we have the following short exact sequence in $\mod \Lambda$
$$0 \rightarrow \tilde{X}_B \rightarrow \Omega_{\Lambda}(X_A) \rightarrow \Omega_A(X_A)\rightarrow 0,$$
where $\tilde{X}_B \in \mod B$, by  
applying Remark \ref{remark idempotentes} for the central column of the diagram. 
Since $\Omega_{\Lambda}(X_A) \subset \rad(P_{\Lambda}(X_A))$, it belongs to $\mod A \oplus \mod B \subset \mod \Lambda$. Finally, using Lemma \ref{lema2}, the result follows.
\end{proof}
\end{lema}

\vspace{0.2cm}

By iterating the syzygy operator using the previous lemma we obtain the following result.

\begin{coro}\label{318c}
	Given $\Lambda_{(0,0)} = \begin{pmatrix}
		A & M \\
		N & B\\
	\end{pmatrix}$ is a Morita context algebra of petite connection type, we have:
\begin{enumerate}
\item $\Omega^n_{A}(X_A) \stackrel{\oplus}{\hookrightarrow} \Omega^n_{\Lambda}(X_A)$ for any $X_A \in \mod A$ and $\Omega^n_{B}(X_B) \stackrel{\oplus}{\hookrightarrow} \Omega^n_{\Lambda}(X_B)$ for any $X_B \in \mod B$;

\item if $X_A \in \mod A$ is periodic as $A$-module, then $X_A$ is periodic as $\Lambda$-module. Analogously for $X_B \in \mod B$;

\item if $X_A \in \mod A$ is virtually periodic as $A$-module, then $X_A$ is virtually periodic as $\Lambda$-module. Analogously for $X_B \in \mod B$.

\end{enumerate}

\end{coro}

\begin{obs} \label{observacion} \

\begin{itemize}
\item In the Lemma \ref{sizigia}, for any $X_A \in \mod A$, the definition of the module $\tilde{X}_B$ it depends only on the $\topp(X_A)$, that is if $Y_A \in \mod A$ is such that $\topp(X_A) \cong \topp(Y_A)$ then $\tilde{X}_B \cong \tilde{Y}_B$. Analogously if $X_B,Y_B \in \mod B $ are such that $\topp(X_B) \cong \topp(Y_B)$, then $\tilde{X}_A \cong \tilde{Y}_A$. Therefore for any $X_A \in \mod A$, we have that $\tilde{X}_B\cong \widetilde{(\topp(X_A))}_B$. Hence we can focus only on semisimple modules.

\item We have $\Omega_{\Lambda}^n(X_A) =  \Omega_A^n(X_A)\ \oplus Y$, where $Y$ is a direct summand of $$( \oplus_{i = 0}^{n-1} \oplus_{a' \in \mathcal{E}_A} \Omega_{\Lambda}^i(\tilde{S}_{a'})) \oplus (\oplus_{i = 0}^{n-1} \oplus_{b' \in \mathcal{E}_B} \Omega_{\Lambda}^i(\tilde{S}_{b'})).$$

Analogously, $\Omega_{\Lambda}^n(X_B) = \Omega_B^n(X_B)\ \oplus Y$, where $Y$ is a direct summand of $$( \oplus_{i = 0}^{n-1} \oplus_{a' \in \mathcal{E}_A} \Omega_{\Lambda}^i(\tilde{S}_{a'})) \oplus (\oplus_{i = 0}^{n-1} \oplus_{b' \in \mathcal{E}_B} \Omega_{\Lambda}^i(\tilde{S}_{b'})).$$

\end{itemize}

\end{obs}

\section{Delooping level on Morita context algebras}

In this section, we study the delooping level ($\Dell$) for Morita context algebras. The central result, Theorem \ref{main theorem}, gives conditions under which the delooping level of the algebra is finite. We discuss also in this section specific cases where these conditions hold, linking the structural properties of syzygies to the finiteness of delooping levels.


Assume for the next two propositions that $A$ and $B$ are indecomposable algebras. This is not a restrictive hypothesis, since any Morita context algebra can be got as an iterated process of Morita context algebras constructed with indecomposable algebras components.
 
 In the first two propositions, the hypotheses provide sufficient conditions to prove that $\Dell$ of the Morita context algebra vanishes.
 
\begin{prop}\label{Dell=0}
	Assume $\Lambda_{(0,0)} = \begin{pmatrix}
		A & M \\
		N & B\\
	\end{pmatrix}$ is a Morita context algebra of petite connection type, such that $A$ and $B$ are non simple algebras, with $\Dell(A) = 0=\Dell(B)$. Then $\Dell(\Lambda_{(0,0)}) = 0$.

\begin{proof}

Consider a simple $\Lambda$-module $S$, then $S = S_a$ with $a \in \mathcal{E}_A $ or  $S=S_b$ with $b \in \mathcal{E}_B$. Let us prove the first case, the second one follows similarly.

\noindent Assume $S=S_a$, with $a \in \mathcal{E}_A $. Since $\Dell(A) = 0$ and $A$ is an indecomposable non simple algebra, then $S_a  \stackrel{\oplus}{\hookrightarrow} \Omega_A(X)$, for some $X \in \mod A$. On the other hand, by Lemma \ref{sizigia}, we have $\Omega_{\Lambda}(X) = \Omega_A(X) \oplus \tilde{X}_B$. Then  $S_a  \stackrel{\oplus}{\hookrightarrow} \Omega_A(X) \stackrel{\oplus}{\hookrightarrow} \Omega_{\Lambda}(X)$, hence $\Dell(S_a) = 0$ as $\Lambda$-module, for all  $a \in \mathcal{E}_A $. Therefore $\Dell (\Lambda_{(0,0)}) = 0$.
\end{proof}

\end{prop}

\begin{prop}\label{Dell=0.2}
	Assume $\Lambda_{(0,0)} = \begin{pmatrix}
		A & M \\
		N & B\\
	\end{pmatrix}$ is a Morita context algebra of petite connection type, such that $M\neq 0$ and $N\neq 0$ as bimodules, with $\Dell(A) = 0 = \Dell(B) = 0$. Then $\Dell(\Lambda_{(0,0)}) = 0$.
	
	\begin{proof}
		
 		 \noindent Consider a simple $\Lambda$-module $S$, then $S = S_a$ with $a \in \mathcal{E}_A $ or  $S=S_b$ with $b \in \mathcal{E}_B$. 
 		 
 		 \noindent If $S=S_a$, $a \in \mathcal{E}_A $, and $A$ is not a simple algebra, 
 		 the result follows from the arguments given at the proof of the last Proposition. In case $A$ is a simple algebra, since $N\neq 0$, then $S_a$ is a direct summand of $\Omega_{\Lambda}(\displaystyle \bigoplus_{b\in \mathcal{E}_B} S_b)$. Then $\Dell(S_a) = 0$ as $\Lambda$-module. In case $S=S_b$, $b \in \mathcal{E}_B$ is analogous.
 		 
	\noindent Therefore $\Dell (\Lambda_{(0,0)}) = 0$.
	\end{proof}	
\end{prop}

The case $M=0=N$ is trivial, so the remaining set of possible hypotheses is analysed in the following observation.

\begin{obs}
	Assume $\Lambda_{(0,0)} = \begin{pmatrix}
		\Bbbk & 0 \\
		N & B\\
	\end{pmatrix}$, that is $\Lambda_{(0,0)}$ is a one point extension algebra. 
	It can be seen that $\Omega^t_{\Lambda}(\Bbbk)=\Omega^{t-1}_{B}(_BN)$, if $t\geq 1$. Therefore $\Dell_{\Lambda}(\Bbbk)=2\!\!- \!\!\Dell_B(_BN) +1$. Then $\Dell(\Lambda_{(0, 0)})= 2\!\!- \!\!\Dell_B(_BN) +1$, if $2\!\!- \!\!\Dell_B(_BN) +1> \Dell (B)$. Otherwise 
	$\Dell(\Lambda_{(0, 0)})\leq \Dell (B)$.
	 
	\noindent In particular, if $\Dell (B)=0$, then 
	$\Dell(\Lambda_{(0, 0)})=
	 \begin{cases}
	 	 2\!\!- \!\!\Dell_B(_BN) +1, & \text{if}\ 2\!\! - \!\!\Dell_B(_BN) +1> 0\\
	 	 0, & \text{otherwise.}
	\end{cases}$\\
	Notice that the Kershaw-Rickard algebra  given in the Example \ref{ejemplo infinito} is of this sort (\cite{KR}).
	
	\end{obs}

The following theorem gives conditions under which the delooping level of the algebra $\Lambda$ is finite. We prove this result under very general hypotheses. The results that follow show the hypothesis are achievable. 

\begin{teo}\label{main theorem} 
	Let $\Lambda_{(0,0)} = \begin{pmatrix}
		A & M \\
		N & B\\
	\end{pmatrix}$ be a Morita context algebra of petite connection type with $\Dell(A) < \infty$ and $\Dell(B)< \infty$.  If there exists $n\in \mathbb{N}$ such that $\Omega_{\Lambda}^n((\oplus_{a \in \mathcal{E}_A}  \tilde{S}_a )\oplus (\oplus_{b \in \mathcal{E}_B}  \tilde{S}_b))$ is infinitely deloopable, then $\Dell(\Lambda_{(0,0)}) < \infty$. 

\begin{proof}

First observe that there is $k \in \mathbb{N}$ such that for $ a \in \mathcal{E}_A$ if $P_a \in \mathcal{P}_A \cap \add \Omega^k_{\Lambda}(\mod \Lambda)$, then $P_a$ is infinitely deloopable in $\Lambda$ (analogously for $b \in \mathcal{E}_B$).  

Consider $S$ a simple $\Lambda$-module, assume $S=S_a$, for $a \in \mathcal{E}_A $ (analogous in case $S=S_b$, for $b \in \mathcal{E}_B$). 

\noindent Since $\Dell(A) < \infty$, then there exists $Z_a\in \mod A$ and $P \in \mathcal{P}_A$ such that $\Omega_A^m(S_a)  \stackrel{\oplus}{\hookrightarrow} \Omega_A^{m+1}(Z_a) \oplus P$. We can consider $m \geq k$.
By Remark \ref{observacion}, we have that $\Omega_{\Lambda}^m(S_a) = \Omega_{A}^m(S_a) \oplus W$, where $W \in$ mod$\Lambda$, verifying the following:

\begin{itemize}

\item $\Omega_A^{m}(S_a)  \stackrel{\oplus}{\hookrightarrow} \Omega_A^{m+1}(Z_a) \oplus P \stackrel{\oplus}{\hookrightarrow} \Omega_{\Lambda}^{m+1}(Z_a) \oplus P$;


\item $W \stackrel{\oplus}{\hookrightarrow} ( \oplus_{i = 0}^{m} \oplus_{a' \in \mathcal{E}_A} \Omega_{\Lambda}^i(\tilde{S}_{a'})) \oplus (\oplus_{i = 0}^{m} \oplus_{b' \in \mathcal{E}_B} \Omega_{\Lambda}^i(\tilde{S}_{b'}))$. 

\end{itemize}
Since $P$ is infinitely deloopable, there is $X \in \mod \Lambda$ such that $P \stackrel{\oplus}{\hookrightarrow} \Omega_{\Lambda}^{m+1}(X) $, then: 
$$\Omega^m_{\Lambda}(S_a) = \Omega_A^m(S_a) \oplus W \stackrel{\oplus}{\hookrightarrow} \Omega_{\Lambda}^{m+1}(Z_a) \oplus P \oplus W  \stackrel{\oplus}{\hookrightarrow} \Omega_{\Lambda}^{m+1}(Z_a) \oplus \Omega^{m+1}(X) \oplus W. $$
By hypothesis $\Omega_{\Lambda}^n( ( \oplus_{i = 0}^{m} \oplus_{a' \in \mathcal{E}_A}
\Omega_{\Lambda}^i(\tilde{S}_{a'})) \oplus (\oplus_{i = 0}^{m} \oplus_{b' \in \mathcal{E}_B} \Omega_{\Lambda}^i(\tilde{S}_{b'})))$ is infinitely deloopable, then, for each $j\in \N$ there exists ${W}_j\in  \mod \Lambda$ such that  $\Omega_{\Lambda}^n(W)  \stackrel{\oplus}{\hookrightarrow}  \Omega_{\Lambda}^{n +j}({W}_j)$. Finally, we have: 

$$\Omega_{\Lambda}^{m+n}(S_a) \stackrel{\oplus}{\hookrightarrow}  \Omega_{\Lambda}^{m+n+1}(Z_a) \oplus \Omega_{\Lambda}^{m+n+1}(X) \oplus \Omega_{\Lambda}^{m+n+1}(W_{m+1})= \Omega_{\Lambda}^{m+n+1}(Z_a\oplus X \oplus W_{m+1}).$$  
Therefore $\Dell (S_a)\leq m+n$, for each $a \in \mathcal{E}_A$. An analogous proof in case $S=S_b$, for $b \in \mathcal{E}_B$.  We can conclude that $\Dell(\Lambda_{(0,0)}) < \infty$.

\end{proof}
\end{teo}

\begin{coro}\label{main coro}
Let $\Lambda_{(0,0)} = \begin{pmatrix}
	A & M \\
	N & B\\
\end{pmatrix}$ be a Morita context algebra of petite connection type such that $\Dell(A) < \infty$ and $\Dell(B)< \infty$. If there is $n \in \mathbb{N}$ such that $\Omega_B^n(\tilde{S}_a)$ and $\Omega_A^n(\tilde{S}_b)$ are infinitely deloopable for each $a \in \mathcal{E}_A$ and for each $b \in \mathcal{E}_B$ as $\Lambda$-modules, then $\Dell(\Lambda) < \infty$.

\begin{proof}

Observe again that there is $k \in \mathbb{N}$ such that for $ a \in \mathcal{E}_A$ if $P_a \in \mathcal{P}_A \cap \add \Omega^k_{\Lambda}(\mod \Lambda)$, then $P_a$ is infinitely deloopable in $\Lambda$ (analogously for $b \in \mathcal{E}_B$).   

If $X_a = \Omega_{B}^{n}(\tilde{S}_a)$ is infinitely deloopable (choose $n \geq k$), then there are $\Lambda$-modules $Z_j$ $\forall j \in \mathbb{N}$ such that $X_a \stackrel{\oplus}{\hookrightarrow} \Omega^j_{\Lambda}(Z_j) \oplus P_B$. Since $n \geq k$, then $P_B$ is infinitely deloopable as $\Lambda$-module. Finally $X_a$ is infinitely deloopable as $\Lambda$-module and we can apply Theorem \ref{main theorem} and the result follows.  
\end{proof}

\end{coro}

The following results shows how can be obtained examples of algebras satisfying Theorem \ref{main theorem} and Corollary \ref{main coro}.

\begin{coro}\label{examples}
Let $\Lambda_{(0,0)} = \begin{pmatrix}
		A & M \\
		N & B\\
	\end{pmatrix}$ be a Morita context algebra of petite connection type such that $\Dell(A) < \infty$ and $\Dell(B)< \infty$. Assume that any of the following statements hold: 

\begin{enumerate}
\item $\tilde{S}_a$ are virtually periodic as $A$-modules (or as $\Lambda$-modules)  for all $a \in \mathcal{E}_A$ and $\tilde{S}_b$ are virtually periodic as $B$-modules (or as $\Lambda$-modules) for all $b \in \mathcal{E}_B$;

\item the class $\{ \Omega_B^n(\tilde{S}_a),\ 
\Omega_A^n(\tilde{S}_b) \}_{a \in \mathcal{E}_A, b \in \mathcal{E}_B, n \in \mathbb{N}}$ is of finite representation type.
\end{enumerate}

Then $\Dell(\Lambda) < \infty$.
\begin{proof}
	\begin{enumerate}
		\item See Remark \ref{31/8} and Corollary \ref{318c}. 
		\item  Use Remark \ref{318r}.
	\end{enumerate} \end{proof}
\end{coro}

A way to obtain a Morita context algebra of petite connection type that satisfies the second item of the previous result is
the following.

\begin{prop}\label{modulos proyectivos}
Let $\Lambda_{(0,0)} = \begin{pmatrix}
	A & M \\
	N & B\\
\end{pmatrix}$ be a Morita context algebra of petite connection type. If $M_B$ is $B$-projective and $N_A$ is $A$-projective, then the class $\mathcal{X}=\{ \Omega_B^n(\tilde{S}_a),\ 
\Omega_A^n(\tilde{S}_b) \}_{a \in \mathcal{E}_A, b \in \mathcal{E}_B, n \in \mathbb{N}}\subset \mathcal{P}_A \cup \mathcal{P}_B$.  In particular $\mathcal{X}$ is of finite representation type.
\begin{proof}

Since $\Lambda = \begin{pmatrix}
A & M \\
0 & 0\\
\end{pmatrix} \oplus 
\begin{pmatrix}
0 & 0\\
N & B\\
\end{pmatrix}$ as a right $\Lambda$-module, then $\begin{pmatrix}
A & M \\
0 & 0\\
\end{pmatrix}$ is a projective right $\Lambda$-module. Then, by the short exact sequences \ref{eq 1}, we have the following short exact sequence of $\Lambda$-modules: 

$$0 \rightarrow \begin{pmatrix}
0 & M \\
0 & 0\\
\end{pmatrix} \rightarrow \begin{pmatrix}
A & M \\
0 & 0\\
\end{pmatrix} \rightarrow A \rightarrow 0$$ 

Since $P_{\Lambda}(A) = A\oplus M$ as $\Lambda$-modules, we conclude
$\tilde{S}_a$ is $B$-projective, for any ${a \in \mathcal{E}_A}$, and $\tilde{S}_b$ is $A$-projective, for any ${b \in \mathcal{E}_B}$.
Therefore $\Omega_{\Lambda}^n(S_a) = \Omega_A^n(S_a) \oplus P_A \oplus P_B$ and $\Omega_{\Lambda}^n(S_b) = \Omega_B^n(S_b) \oplus P'_B \oplus P'_A$ for any $n\in \N$, where $P_A, P'_A \in \mathcal{P}_A$, $P_B, P'_B \in \mathcal{P}_B$. Hence $\mathcal{X}$ is of finite representation type.
\end{proof}

\end{prop}

The Proposition below shows another way to create examples for Theorem \ref{main theorem}. In particular Example \ref{example 1} fulfills these conditions.

\begin{prop} \label{combi}
Let $A = \frac{\Bbbk Q_A}{I_A}$ and $B = \frac{\Bbbk Q_B}{I_B}$ be finite dimensional algebras. Consider $\Lambda = \frac{\Bbbk Q}{I}$, with:
\begin{itemize}
	\item[*] $Q_0 = (Q_A)_0 \sqcup (Q_B)_0$,
	\item[*] a set of arrows $\Gamma_1$ such that if $\delta \in \Gamma_1$ then $\st(\delta) \in (Q_A)_0 $ and $\tg(\delta) \in (Q_B)_0 $ or $\st(\delta) \in (Q_B)_0 $ and $\tg(\delta) \in (Q_A)_0 $,
	\item[*]  $Q_1 = (Q_A)_1 \sqcup (Q_B)_1 \sqcup \Gamma_1$, and 
	
	\item[*] $I = \langle I_A, I_B, \gamma\delta , \forall \gamma \in Q_1,\forall \delta \in \Gamma_1  \rangle $.

\end{itemize}

\noindent Consider $\Gamma'_0 = \{\st(\delta),\  \delta \in \Gamma_1 \}$, and $\Gamma'_1 = \{\delta \in \Gamma_1, \ \tg(\delta)\in \Gamma'_0 \}$. Suppose that for each $v_0 \in \Gamma'_0$ there is an arrow  $\rho \in \Gamma_1$ where $\tg(\rho) = v_0$. If $\Dell(A) < \infty$ and $\Dell (B) < \infty$, then $\Dell(\Lambda) < \infty$. 

\begin{proof}
By the definition of the ideal $I$, we can prove that $\tilde{S}_a$ and $\tilde{S}_b$ are infinitely deloopable as $\Lambda$-modules for all $a \in (Q_A)_0$ and for all $b \in (Q_B)_0$ respectively:\\
\noindent Suppose $a \in (Q_A)_0$, then we have two cases to consider: $\tilde{S}_a = 0$, or $\tilde{S}_a \not = 0$.

\noindent For the first case, it is clear that $\tilde{S_a} = 0$ is infinitely deloopable. For the second case, we have, by hypothesis, that there is an arrow $\gamma_0 \in \Gamma_1$ such that $\st(\gamma_0) = b_0$ and $\tg(\gamma_0) = a$. As a consequence $S_a$ is a direct summand of $\topp(\tilde{S}_{b_0})$, then $\tilde{S}_a  \stackrel{\oplus}{\hookrightarrow} \Omega_{\Lambda}^2(S_{b_0})$. Since $\gamma_0 \in \Gamma_1$ then $b_0\in \Gamma'_0$, we also have $\gamma_1 \in \Gamma_1$ such that $\tg(\gamma_1) = b_0$ and $\st(\gamma_1) = a_1$, then $S_{b_0}$ is a direct summand of $\topp(\tilde{S}_{a_1})$, then $\tilde{S}_{b_0}  \stackrel{\oplus}{\hookrightarrow} \Omega^2(S_{a_1})$, and $\tilde{S}_a  \stackrel{\oplus}{\hookrightarrow} \Omega^4(S_{a_1})$. Inductively, we get that $\tilde{S}_{a}$ is infinitely deloopable.  Hence the result follows from Theorem \ref{main theorem}.

\end{proof}

\end{prop}

\section{Other delooping levels}

This section extends the analysis to other variants of delooping:  the derived delooping level ($\DDell$) and the global delooping level ($\Dellm $). One of the aims of this section is to show how the finiteness of these invariants in the components $A$ and $B$ of a Morita context implies their finiteness in the whole algebra $\Lambda$.

\subsection{Results for the derived delooping level on Morita context algebras}

The result below follows by Remark \ref{observacion}, by imitating the proof of Theorem \ref{main theorem}.

\begin{prop}\label{kDell algebras de morita}	
	Let $\Lambda_{(0,0)} = \begin{pmatrix}
	A & M \\
	N & B\\
\end{pmatrix}$ be a Morita context algebra of petite connection type. Assume there exists $n\in \mathbb{N}$ such that $\Omega_{\Lambda}^n((\oplus_{a \in \mathcal{E}_A}  \tilde{S}_a )\oplus (\oplus_{b \in \mathcal{E}_B}  \tilde{S}_b))$ is infinitely deloopable, and let $X \in \mod A \ (\mod B)$. If $k-\Dell_A(X) < \infty \ (k-\Dell_B(X) < \infty)$, then $k-\Dell_{\Lambda}(X) < \infty$.
		
\end{prop}

\begin{prop}\label{DDell=0}
For an Artin algebra $\Gamma$, $\DDell(\Gamma) = 0$ if and only if $\Dell(\Gamma)=0$.
\begin{proof}

By Theorem \ref{GUO&IGUSA}, if $\Dell(\Gamma) = 0$ then $\DDell(\Gamma)=0$. Conversely, if $M \in \mod \Gamma$ and $\DDell(M) = 0$, there is an exact sequence $0 \rightarrow C_0 \rightarrow M \rightarrow 0$ such that $\Dell(C_0) = 0$. Since $C_0 \cong M$, then $\Dell(M) = 0$.
\end{proof}

\end{prop}

As a consequence of the above result, we have a variant of Proposition \ref{Dell=0} for the derived delooping level. 

\begin{prop}
Assume $\Lambda_{(0,0)} = \begin{pmatrix}
	A & M \\
	N & B\\
\end{pmatrix}$ is a Morita context algebra of petite connection type
\begin{enumerate}
	\item 
If $A$ and $B$ are non simple algebras, with $\DDell(A) = 0=\DDell(B)$, then $\DDell(\Lambda_{(0,0)}) = 0$.
	
	\item 
 If $M\neq 0$ and $N\neq 0$ as bimodules, and $\DDell(A) = 0 = \DDell(B) = 0$, then $\DDell(\Lambda_{(0,0)}) = 0$.

\end{enumerate}

\end{prop}

\begin{coro}\label{Ddel algebras de morita}
	Let $\Lambda_{(0,0)} = \begin{pmatrix}
	A & M \\
	N & B\\
\end{pmatrix}$ be a Morita context algebra of petite connection type. Assume there exists $n\in \mathbb{N}$ such that $\Omega_{\Lambda}^n((\oplus_{a \in \mathcal{E}_A}  \tilde{S}_a )\oplus (\oplus_{b \in \mathcal{E}_B}  \tilde{S}_b))$ is infinitely deloopable. If $\DDell(A) < \infty$ and $\DDell(B) < \infty$, then $\DDell(\Lambda) < \infty$.
\begin{proof}
Let $S$ be a simple $\Lambda$-module. Assume  $S \in \mod A$ (the other case is similar). Since $\DDell(A) < \infty$, then there is a exact sequence in $\mod A$: $0 \rightarrow C_n \rightarrow C_{n-1} \rightarrow \ldots \rightarrow C_1 \rightarrow C_0 \rightarrow S \rightarrow 0$, such that $(i+1)-\Dell_{A} (C_i) < \infty,\ \forall i = 0, 1, \ldots , n$. By Corollary \ref{kDell algebras de morita}, we have $(i+1)-\Dell_{\Lambda} (C_i) < \infty,\ \forall i = 0, 1, \ldots , n$. Hence we deduce that $\DDell_{\Lambda}(S) < \infty$.
\end{proof}
	
\end{coro}

\subsection{Results for the global delooping level on Morita context algebras}
\begin{prop}

Let $\Lambda_{(0,0)} = \begin{pmatrix}
		A & M \\
		N & B\\
	\end{pmatrix}$ be a Morita context algebra of petite connection type with $\Dellm(A) < \infty$ and $\Dellm(B)< \infty$. If $M_B$ and $N_A$ are projective, then $\Dellm(\Lambda) < \infty$.

\begin{proof}

Consider $X \in \mod \Lambda$. By Remark \ref{radical}, since $M_B$, $N_A$ are projective as $B$-module and $A$-module respectively, we have that $\Omega_{\Lambda}(X) \hookrightarrow A^{k_1} \oplus B^{k_2}$. By Lemma \ref{lema1}, we deduce that $\Omega_{\Lambda}(X) = Y_1 \oplus Y_2$ such that $Y_1 \hookrightarrow A^{k_1}$ and $Y_2 \hookrightarrow B^{k_2}$, then there exist $X_1\in \mod A$ and $X_2 \in \mod B$ such that $\Omega_A(X_1)=Y_1\oplus P_A$ and $\Omega_B(X_2)=Y_2 \oplus P_B$, with $P_A \in \mathcal{P}_A$ and $P_B\in \mathcal{P}_B$. Then  $\Omega_{\Lambda}(X) = \Omega_A(X_1) \oplus \Omega_B(X_2) \oplus P_A \oplus P_B$, where $P_A \in \mathcal{P}_A$ and $P_B\in \mathcal{P}_B$.\\

On the other hand, by Proposition \ref{modulos proyectivos} and the first item of Remark \ref{observacion}, we have that for any $W\in \mod A$, $\Omega_{\Lambda}^n(W) = \Omega_A^n(W) \oplus Q_A \oplus Q_B$, with $Q_A \in \mathcal{P}_A$ and $Q_B\in \mathcal{P}_B$, and for any $Z\in \mod B$, $\Omega_{\Lambda}^n(Z) = \Omega_B^n(Z) \oplus Q'_A \oplus Q'_B$, with $Q_A$, $Q'_A$ projective $A$-modules, and 
$Q_B$, $Q'_B$ projective $B$-modules.\\

Using arguments similar to those set out in the proof of Theorem \ref{main theorem}, the thesis follows. 


 \end{proof}

\end{prop}

\section{Examples}

This final section provides explicit examples illustrating the theoretical results. Some examples show that the delooping level of a Morita context algebra $\Lambda$ can be smaller than that of its components $A$ and $B$, while others highlight the necessity of the assumptions in the main theorems. These examples clarify the scope and limitations of the results and connect theory with concrete cases.\\

First observe that Example \ref{ejemplo infinito} is under the hypothesis of Lemma \ref{sizigia} where ${Q_A}_0 = \{1\}$ and ${Q_B}_0 = \{2\}$. It shows it is necessary to add extra conditions to guarantee the finiteness of $\Dell(\Lambda)$ (see for instance \ref{main theorem}).\\  

The following examples show that we cannot obtain a bound for the (derived) delooping levels of $A$ and $B$ from the (derived) delooping level of $\Lambda$, even under the hypothesis of Theorem \ref{main theorem}. 

\begin{ej}\label{example 1}

Consider the algebra $\Lambda = \frac{\Bbbk Q}{F^2}$, where $Q$ is the quiver:

$$\xymatrix{1 \ar[r]^{\alpha_1} \ar@/^{3mm}/[d]^{\lambda} & 2 \ar[r]^{\alpha_2}    & \ldots \ar[r] & n-1 \ar[r]^{\alpha_{n-1}}  & n \\ 1' \ar@/^{3mm}/[u]^{\mu} \ar[r]^{\beta_1}  & 2' \ar[r]^{\beta_2} & \ldots  \ar[r] & m'-1 \ar[r]^{\beta_{m'-1}} & m'}$$


Consider $A = \frac{\Bbbk Q_A}{I_A}$ and $B = \frac{\Bbbk Q_B}{I_B}$ where $Q_A$ is the full subquiver of $Q$ generated by the vertex $\{1, \ldots, n\}$ and $Q_B$ is the full subquiver of $Q$ generated by the vertex $\{1', \ldots, m'\}$, $I_A = F^2_A$, and $I_B = F^2_B$.\\

Then $\Lambda \cong \begin{pmatrix} A & M \\ N & B \end{pmatrix}$ is a Morita context algebra under the hypothesis of Theorem  \ref{main theorem} and Corollary \ref{Ddel algebras de morita}. Therefore $\Dell(\Lambda)$ and
$\DDell(\Lambda)$ are finite.\\

More than that, since $A$ and $B$ are radical square zero algebras, then $\Fin(A^{op})=\fin(A^{op}) = n-1$ and $\Fin(B^{op})=\fin(B^{op}) = m'-1$, by item 2 of Theorem \ref{Igualdad}, $\Dell(A) = n-1$ and $\Dell(B) = m'-1$, and by Theorem \ref{GUO&IGUSA} $\DDell(A) = n-1$ and $\Dell(B) = m'-1$, However, since every simple $\Lambda$-module is a submodule of $\Lambda$, we have $\Dell(\Lambda) =0$. And, again by Theorem \ref{GUO&IGUSA}, we have $\DDell(\Lambda) = 0$. 


\end{ej}

\begin{ej}\label{example 2}
Consider the algebra $\Lambda = \frac{\Bbbk Q}{I}$, where $Q$ is the quiver:

$$ \xymatrix{1  \ar[r]^{\delta} \ar@/^{3mm}/[dd]^{\lambda_1}   & 2 \ar@(ru,rd)^{\beta} \ar@(lu,ru)^{\alpha} \ar@(ld,rd)_{\gamma} \\ & \\ 1'  \ar[r]^{\delta'} \ar@/^{3mm}/[uu]^{\lambda_2}  & 2' \ar@(ru,rd)^{\beta'} \ar@(lu,ru)^{\alpha'} \ar@(ld,rd)_{\gamma'} }$$

\noindent and the ideal $I = \langle I_q, I'_q, \lambda_i \lambda_j,\text{ for } i\not=j,\ \lambda_1 \delta',\ \lambda_2 \delta \rangle$ being $I_q$ and $I'_q$ are the same ideal as in Example \ref{ejemplo infinito} for the full subquivers generated by the vertices $1, 2$ and $1',2'$ respectively. It is clear that $\Lambda$ is a Morita context algebra, with $A$ and $B$ Artin algebras with $\Dell(A)=\Dell(B)=\infty$. Since $S_1 \subset \soc( P_{1'})$, $S_{1'} \subset \soc( P_1)$, $S_2 \subset \soc(P_2)$, and $S_2' \subset \soc(P_2')$, we have that $\Dell(S_1) = \Dell(S_{1'}) = \Dell(S_2) = \Dell(S_{2'}) = 0$. Hence $\Dell(\Lambda) = 0$. 

\end{ej}

The following example shows that in Proposition \ref{Dell=0}, we cannot put $\Dellm$ instead of $\Dell$.

\begin{ej}
Let $\Lambda = \frac{\Bbbk Q}{F^2}$ be a radical square algebra with $Q$ as follows
$$ \xymatrix{ 1 \ar@(dl,lu)^{\alpha} \ar@/^/[r]^{\beta}& 2 \ar@/^/[l]^{\gamma} \ar@(dr,ru)_{\delta} }$$
The algebra $\Lambda$ can be seen as a Morita context algebra as we show below:

$$\Lambda = \begin{pmatrix}
		A & S \\
		S & A\\
	\end{pmatrix}$$ 
where $A = \frac{\Bbbk [x]}{\langle x^2 \rangle }$ and $_AS$, $S_A$ are both simple modules.	
It is clear that $A$ is selfinjective, however $\Lambda$ is not selfinjective. That is $\Dellm(A)=0$ and $\Dellm(\Lambda)\neq 0$, using Corollary 3.16 of \cite{BLM}.\\
\end{ej}

\noindent {\bf Statement:} The authors declare that this manuscript is \textbf{\textit{AI free}}.

\end{document}